\documentclass{amsart}
\usepackage[english]{babel}
\usepackage{amsfonts, amsmath, amsthm, amssymb,amscd,indentfirst}

\usepackage{esint}
\usepackage{graphicx}
\usepackage{epstopdf}
\usepackage{amsmath,amssymb,latexsym,indentfirst}
\usepackage{times}
\usepackage{palatino}

\newtheorem{theorem}{Theorem}[section]

\newtheorem{lemma}{Lemma}[section]

\newtheorem{corollary}{Corollary}[section]

\def\bp{\begin{proof}}
\def\ep{\end{proof}}

\begin{document}

\title[On stable CMC hypersurfaces with free-boundary in a Euclidean Ball]{On stable CMC hypersurfaces with free-boundary in a Euclidean Ball}

\author{Ezequiel Barbosa}
\address{Universidade Federal de Minas Gerais (UFMG), Departamento de Matem\'{a}tica, Caixa Postal 702, 30123-970, Belo Horizonte, MG, Brazil}
\email{ezequiel@mat.ufmg.br}
\thanks{The author partially supported by CNPq-Brazil.}

\begin{abstract}
 In this note, we observe that if $B$ is a ball in a Euclidean space with dimension $n$, $n\geq3$, then a stable CMC hypersurface $\Sigma$ with free boundary in $B$ satisfies
 \[
nA\leq L\leq nA\left( \frac{1+\sqrt{1+4(n+1)H^2}}{2} \right)\,, 
 \]
where $L$, $A$ and $H$ denote the length of $\partial \Sigma$, the area of $\Sigma$ and the mean curvature of $\Sigma$, respectively. Consequently, if the boundary   
$\partial \Sigma$ is embedded then $\Sigma$ must be totally geodesic or starshaped with respect to the center of the ball. This result is an improvement of a theorem proved by A. Ros and E. Vergasta \cite{R-V} . In particular, if $n=3$, the only stable CMC surfaces with free boundary in $B$ are the totally geodesic disks or the spherical caps. This last result was proved very recently by I. Nunes \cite{N} using an extended stability result and a modified Hersch type balancing argument to get a better control on the genus. We don't use that modified Hersch type argument. However, we use a Nunes type Stability Lemma and a crucial result due to A. Ros and E. Vergasta.
\end{abstract}

\maketitle

\section{Introduction}\label{intro}
Given a smooth compact and convex domain $B$ in $\mathbb{R}^{n+1}$, denote by $\partial B$ and $int\,B$ the boundary and the interior of $B$, respectively. A CMC free-boundary hypersurface in $B$ is a constant mean curvature hypersurface $\Sigma \subset B$ meeting $\partial B$ orthogonally along $\partial \Sigma$. That kind of hypersurfaces are solutions for the problem of finding critical points of the area functional among all compact hypersurfaces $\Sigma \subset B$ with $\partial \Sigma \subset \partial B$ which divides $B$ into two subsets of prescribed volumes. If a CMC free-boundary hypersurface $\Sigma \subset B$  has nonnegative second variation of area for all preserving volume variations we name it as a CMC free-boundary stable hypersurface. For more details about CMC free-boundary hypersurfaces, see the following references and references therein: \cite{N}, \cite{R}, \cite{R-V}, \cite{So}.

In \cite{R-V}, Ros and Vergasta studied stable CMC hypersurfaces with free boundary when $B$ is a ball and proved the following result. Denote by $L$ the length of the boundary $\partial \Sigma$ and by $A$ the area of $\Sigma$. 

\begin{theorem}[Ros-Vergasta \cite{R-V}]
Let $B\subset\mathbb{R}^n$, $n\geq3$, be a closed ball. Let $\Sigma \subset B$ be a CMC free-boundary stable hypersurface with embedded boundary in $B$. If $L\geq nA$ then $\Sigma$ is totally geodesic or starshaped with respect to the center of the ball.
\end{theorem}

In order to improve the result above, we use a Nunes type Stability Lemma (see Lemma \ref{L}) to prove that always $L\geq nA$. More precisely, we obtain the following result.

\begin{theorem}\label{main}
Let $B\subset\mathbb{R}^n$, $n\geq3$, be a closed ball. If $\Sigma \subset B$ is a CMC free-boundary stable hypersurface in $B$,
then 
\[
nA\leq L\leq nA\left( \frac{1+\sqrt{1+4(n+1)H^2}}{2} \right)\,.
\]
In particular, if $\partial \Sigma$ is embedded, then $\Sigma$ is totally geodesic or starshaped with respect to the center of the ball.
\end{theorem}

As a direct consequence, we obtain the following corollary.
\begin{corollary}
Let $B\subset\mathbb{R}^n$, $n\geq3$, be a closed ball. If $\Sigma \subset B$ is a CMC stable hypersurface with embedded free-boundary in $B$ and $0\in \Sigma$ then $\Sigma$ is totally geodesic.
\end{corollary}

Now, noting that free-boundary surfaces must have embedded boundary (see Theorem 11  in \cite{R-V}), we obtain a complete topological classification.

\begin{corollary}\label{COR}
Let $B\subset\mathbb{R}^3$ be a closed ball. If $\Sigma \subset B$ is a CMC stable surface with free-boundary then $\Sigma$ is a totally geodesic disk or a spherical cap. 
\end{corollary}

Its worthy to mention that the Corollary \ref{COR} was proved recently by I. Nunes \cite{N} using a powerful stability result and a modified Hersch type balancing argument to get a better control on the genus and on the number of connected components of the boundary of the surface. In fact, I. Nunes proved a more general result which, joint with Theorem 11  in \cite{R-V}, gives us the result above as a corollary.

\begin{theorem}[I. Nunes \cite{N}]
Let $\Omega\subset \mathbb{R}^3$ be a smooth compact convex domain. Suppose that the second fundamental form $\Pi^{\partial \Omega}$ of $\partial \Omega$ satisfies the pinching condition
\[
k\,h\leq \Pi^{\partial \Omega}\leq (3/2) k\,h\,.
\]
for some constant $k > 0$, where $h$ denotes the induced metric on $\partial \Omega$. If $\Sigma \subset \Omega$ is an immersed orientable compact stable CMC surface with free boundary, then $\Sigma$ has genus zero and $\Sigma$ has at most two connected components.
\end{theorem}

In order to prove Theorem \ref{main}, we apply the same idea as that applied by I. Nunes in the proof of the main result for the free-boundary surfaces case in \cite{N}. I. Nunes showed that the stability of a 
free-boundary CMC surface implies that the quadratic form
given by the second variation of area is nonnegative for all functions $f$
such that $f=0$ on $\partial\Sigma$ regardless of whether it satisfies $\int_{\Sigma}fdvol_{\Sigma}=0$ or
not. That is what we are calling Nunes Stability Lemma. Then I. Nunes was able to apply a modified Hersch type balancing argument to obtain a better control on the genus of $\Sigma$. We use that idea for high dimension combined with some Ros-Vergasta results.

\section{Nunes Type Stability Lemma}

Let $B$ be a compact convex domain in $\mathbb{R}^{n+1}$. Let $\varphi : \Sigma^n \rightarrow B$ be an immersion  of a smooth orientable manifold $\Sigma$ with boundary $\partial \Sigma$ such that $\varphi(\partial \Sigma)=\varphi(\Sigma)\cap \partial B$. Let's denote the unit normal vector of the hypersurface $\Sigma$ by $N$.  The immersion $\varphi$ is called free boundary if $\varphi(\Sigma)$ meets $\partial B$ orthogonally. The second fundamental form $A$ of $\Sigma$ is the endomorphim $A(X)=-\nabla_X N$, where $X \in T\Sigma$. The mean curvature of $\Sigma$ is then given by $H=\frac{1}{n}\text{Tr}A$.

If we consider a smooth variation $\phi: \Sigma\times [0,\varepsilon) \rightarrow B$ that preserves $\partial B$ and such that $\phi(\cdot, 0)=\varphi(\cdot)$ then it natural to consider the following two functions: 
\[A(t)=\int_{\Sigma} dvol_{\Sigma_t} \quad \text{and}\quad V(t)=\int_{\Sigma \times [0,t]} \phi^*dvol_{\mathbb{R}^{n+1}}. \]
The variation $\phi$ is called volume preserving if $V(t)\equiv 0$.
Let $f$ be the function defined by $f= \langle \frac{\partial}{\partial t}\phi(x, 0), N(x)\rangle$ where $x \in \Sigma$ then the first variation formula yields: 
\[A^{\prime}(0)=- n \int_{\Sigma} H f dvol_{\Sigma}+\int_{\partial \Sigma} f \langle N, \nu \rangle ds \quad 
\text{and} \quad V^{\prime}(0)= \int_{\Sigma}f dvol_{\Sigma}.\]
It follows that CMC hypersurfaces with free boundary are critical points for the area functional $A(t)$ when restricted to volume preserving variations. The converse is also true, see \cite{R-V}. If $A^{\prime \prime}(0)\geq 0$ is nonnegative for every volume preserving variation then the immersion $\phi$ is called Stable CMC. It can be shown that this is equivalent to have for every $f \in C^{\infty}(\Sigma)$ with $\int_{\Sigma}f dvol_{\Sigma}=0$ that
\[I(f,f)= \int_{\Sigma} |\nabla f|^2- |A_{\Sigma}|^2f^2 dvol_{\Sigma} - \int_{\partial \Sigma}\Pi(N,N)f^2 ds \geq 0\]

\begin{lemma}[Nunes Type Stability Lemma]\label{L}
Let $\Sigma$ be an immersed stable hypersurface with constant mean curvature with free boundary in $B$. If $f \in C^{\infty}(\Sigma)$ is such that $f(x)=0$ for every $x \in \partial \Sigma$ then
\begin{eqnarray*}
I(f,f)= \int_{\Sigma} |\nabla f|^2- |A_{\Sigma}|^2 f^2 dvol_{\Sigma}\geq \frac{1}{n+1}\left( \frac{\int_{\Sigma}fdvol_{\Sigma}}{A} \right)^2\int_{\partial\Sigma}\Pi(N,N)ds.
\end{eqnarray*}
In particular, If $f \in C^{\infty}(\Sigma)$ is such that $f(x)=0$ for every $x \in \partial \Sigma$ then
\begin{eqnarray*}
I(f,f)= \int_{\Sigma} |\nabla f|^2- |A_{\Sigma}|^2 f^2 dvol_{\Sigma}\geq 0\,.
\end{eqnarray*}
\end{lemma}
\begin{proof}
Let $f_i$ be the function $f_i= \langle e_i, N \rangle$ where $
\{e_i\}$ is the canonical orthonormal basis of $\mathbb{R}^{n+1}$. A simple computation yields:
\begin{eqnarray}\label{jacobi function}
\Delta f_i + |A_{\Sigma}|^2 f_i=0.\end{eqnarray}
Plugging these functions on the quadratic form $I$ we have:
\begin{eqnarray*}
\sum_{i=1}^{n+1} I(f_i,f_i)&=&\int_{\Sigma}\sum_{i=1}^{n+1} |\nabla f_i|^2 - |A_{\Sigma}|^2f_i^2 dvol_{\Sigma} - \int_{\partial \Sigma}\sum_{i=1}^{n+1}\Pi(N,N)f_i^2 ds\\
&=& - \sum_{i=1}^{n+1} \int_{\Sigma} f_i\Delta f_i + |A_{\Sigma}|^2 f_i dvol_{\Sigma} + \int_{\partial \Sigma} \sum_{i=1}^{n+1}f_i \frac{\partial f_i}{\partial \nu} - \Pi(N,N) ds \\
&=& \int_{\partial \Sigma}\frac{1}{2}\frac{\partial}{\partial \nu}(\sum_{i=1}^{n+1} f_i^2)ds - \int_{\partial \Sigma}\Pi(N,N) ds=-\int_{\partial \Sigma}\Pi(N,N) ds.
\end{eqnarray*}
We have used that $\sum_{i=1}^{n+1}f_i^2= |N|^2=1$. It follows that, given a function $f$ such that $f=0$ on $\partial \Sigma$, at least one of the $f_i$ have the property that 

\[
I(f_i,f_i) \leq - \frac{1}{n+1}\int_{ \partial \Sigma}\Pi(N,N)ds < 0 \quad \text{and}\quad f_i\neq f. 
\] 
In fact, if for each $f_i$ we have $I(f_i,f_i) > - \frac{1}{n+1}\int_{ \partial \Sigma}\Pi(N,N)ds$ or $f_i=f$ then we obtain that 
\[
\sum_{i=1}^{n+1} I(f_i,f_i)> - \frac{m}{n+1}\int_{\partial \Sigma}\Pi(N,N) ds
\]
for some positive integer $m\leq\ n+1$, since when $f_i=f$ we obtain $I(f_i,f_i)=0$. This gives us the contradiction 
\[
-\int_{\partial \Sigma}\Pi(N,N) ds=\sum_{i=1}^{n+1} I(f_i,f_i)> - \frac{m}{n+1}\int_{\partial \Sigma}\Pi(N,N) ds\,.
\]
Hence, let $f_i$ be the function satisfying that condition. Note that, because of the stability of $\Sigma$, we have that $\int_{\Sigma}f_idvol_{\Sigma}\neq 0$. Assume that $\int_{\Sigma}fdvol_{\Sigma}\neq0$. Now, consider the function $\bar{f}=c f$, where
\[
c=\frac{\int_{\Sigma}f_idvol_{\Sigma}}{\int_{\Sigma}fdvol_{\Sigma}}\,.
\]
We have  $\int_{\Sigma}( \bar{f} - f_i) \,dvol_{\Sigma}=0$. Using (\ref{jacobi function}) and that $\bar{f}=0$ at $\partial \Sigma$ we have
\[0 \leq  I(\bar{f}-f_i,\bar{f}-f_i)= I(\bar{f},\bar{f}) - 2 I(\bar{f},f_i) + I(f_i,f_i)\leq I(\bar{f},\bar{f})-\frac{1}{n+1}\int_{\partial\Sigma}\Pi(N,N)ds.\]
This implies that
\[
I(f,f)\geq\left(\frac{\int_{\Sigma}fdvol_{\Sigma}}{\int_{\Sigma}f_idvol_{\Sigma}}\right)^2\frac{1}{n+1}\int_{\partial \Sigma}\Pi(N,N) ds\,.
\]
It follows from Holder's inequality and $\sum_{i=1}^{n+1}f_i^2= |N|^2=1$ that
\[
\left( \int_{\Sigma}f_i dvol_{\Sigma}\right)^2=\left|\int_{\Sigma}f_i dvol_{\Sigma}\right|^2\leq\left( \int_{\Sigma}|f_i| dvol_{\Sigma}\right)^2\leq A^2\,.
\]
This finishes the proof.
\end{proof}

\section{Proof of Theorem \ref{main}}
\begin{proof}
Assume that $\Sigma$ is a stable free-boundary hypersurface in $B$. Consider then the support function
$u=<\psi,N>$ of $\Sigma$, where $\psi$ is the immersion of $\Sigma$ in $B$. It satisfies the following
\[
\left\{
  \begin{array}{ccccccc}
    \Delta u & + & |\sigma|^2u & = & -nH & on & \Sigma  \\
    u & = & 0 &  &  & on & \partial\Sigma \\
  \end{array}
\right.
\]
Moreover, taking the diverge of the tangent component $\psi-uN$ of $\psi$ is given by 
\[
div(\psi-uN)=n+nHu\,.
\]
It follows from the Divergence Theorem that
\begin{equation}\label{e1}
L=n\left(A+\int_{\Sigma}Hudvol_{\Sigma}  \right)\,.
\end{equation}
Since $u=0$ on $\partial \Sigma$, it follows from the stability of $\Sigma$ and Nunes Stability Lemma that
\[
nH\int_{\Sigma}udvol_{\Sigma}=\int_{\Sigma}|\nabla_{\Sigma} u|^2-|A_{\Sigma}|^2u^2dvol_{\Sigma}\geq0\,.
\]
Note that, if $H=0$, then $L=nA$. Assume that $H\neq 0$. First, as was done by Ros-Vergasta in \cite{R-V}, we will first prove that either $u\geq0$ or $u\leq0$ on $\Sigma$. Suppose, 
by contradiction, that $u$ changes sign. Consider $\Sigma^+$ (resp. $\Sigma^-$) the subset of $\Sigma$ where $u$ is positive (resp. negative)
and define $u^+$, $u^- \in H^1(\Sigma)$ by
\[
u^+(p)=\left\{
  \begin{array}{ccccc}
    u(p) & if & p & \in & \Sigma^+ \\
    0 & if & p & \in & \Sigma\setminus\Sigma^+ \\
  \end{array}
\right.
 and \,\,\,
u^-(p)=\left\{
  \begin{array}{ccccc}
    u(p) & if & p & \in & \Sigma^- \\
    0 & if & p & \in & \Sigma\setminus\Sigma^- \\
  \end{array}
\right.
\]
A direct computation gives 

\[
I(u^-,u^-)=nH\int_{\Sigma}u^-dvol_{\Sigma}  
\]
and
\[
I(u^+,u^+)=nH\int_{\Sigma}u^+dvol_{\Sigma}\,.
\]
Now we define $\tilde{u}=u^++au^-$, where $a$ is a positive constant such that $\int_{\Sigma}\tilde{u}dvol_{\Sigma}=0$. It follows that $\tilde{u}$ is not identically null and
\[
I(\tilde{u},\tilde{u})=-naH\int_{\Sigma}udvol_{\Sigma}\,.
\]
As in Ros-Vergasta \cite{R-V}, pag. 29, we obtain that either $u\geq0$ or $u\leq0$ on $\Sigma$. We can choose the orientation on $\Sigma$ such that $u\geq0$. Since $H\neq0$ and $\int_{\Sigma}HudA \geq0$, we get that $H>0$. Therefore, $u$ satisfies: $u\geq0$, $u=0$ on $\partial \Sigma$ and $ \Delta u = |\sigma|^2u-nH<0$. By the maximum principle for subharmonic functions we obtain that $u$ is strictly positive on $int \Sigma$.  This gives us that $\int_{\Sigma}udvol_{\Sigma}\neq0$. It follows from the Nunes 
Stability Lemma that
\begin{eqnarray*}
\int_{\Sigma}|\nabla_{\Sigma} u|^2-|A_{\Sigma}|^2u^2dvol_{\Sigma} &\geq&\frac{1}{n+1}\left( \frac{\int_{\Sigma}udvol_{\Sigma}}{A} \right)^2\int_{\partial\Sigma}\Pi(N,N)ds\\
&=&\frac{L}{n+1}\left( \frac{\int_{\Sigma}u\,dvol_{\Sigma}}{A} \right)^2\,.
\end{eqnarray*}
Hence, we obtain
\begin{equation}\label{e2}
n\int_{\Sigma}Hu\,dvol_{\Sigma} \geq\frac{L}{n+1}\left( \frac{\int_{\Sigma}u\,dvol_{\Sigma}}{A} \right)^2\,,
\end{equation}
since
\[
n\int_{\Sigma}Hu\,dvol_{\Sigma} = \int_{\Sigma}|\nabla_{\Sigma} u|^2-|A_{\Sigma}|^2u^2dvol_{\Sigma} \,.
\]
From (\ref{e1}) we have that
\[
 \frac{\int_{\Sigma}u\,dvol_{\Sigma}}{A} =\frac{L-nA}{nHA}\,.
\]
Then, from (\ref{e1}) and (\ref{e2}), we conclude that
\[
L=nA+nH\int_{\Sigma}u\,dvol_{\Sigma}\geq nA+\frac{L}{n+1}\left( \frac{\int_{\Sigma}u\,dvol_{\Sigma}}{A} \right)^2= nA+\frac{L}{n+1}\left( \frac{L-nA}{nHA} \right)^2\,.
\]
This implies that
\[
L-nA\geq \frac{L}{n+1}\left( \frac{L-nA}{nHA} \right)^2\,.
\]
Therefore,
\[
L^2-nA\,L-n^2A^2(n+1)H^2\leq0\,.
\]
This implies that
\[
L\leq nA\left( \frac{1+\sqrt{1+4(n+1)H^2}}{2} \right)\,.
\]

\end{proof}

\section{Proof of Corollary \ref{COR}}
\begin{proof}
As in the proof of Theorem 11  in \cite{R-V}, we obtain that $\partial \Sigma$ is embedded. Now, applying the Theorem \ref{main}, we obtain that $\Sigma$ is totally geodesic or starshaped with respect to the center of the ball. Since starshaped surfaces must have genus 0, we obtain that $\Sigma$ is totally geodesic or a spherical cap. 
\end{proof}

\end{document}